\documentclass[a4paper, 12pt, reqno]{amsart}

\usepackage[normalem]{ulem} 
\usepackage{lmodern}
\usepackage[english]{babel}
\usepackage[latin1]{inputenc}
\usepackage{microtype}
\usepackage{amsfonts,amssymb,amsmath,amsthm,mathrsfs}
\usepackage{mathtools}
\usepackage{color, graphicx}
\usepackage[all]{xy}
\usepackage[bookmarks=false,colorlinks,urlcolor=cyan,citecolor=red,linkcolor=blue]{hyperref}  
\usepackage[inline]{enumitem}
\usepackage{cleveref}
\usepackage{a4wide}
\usepackage{upgreek}
\usepackage{etoolbox}

\makeatletter
\patchcmd{\@setaddresses}{\indent}{\noindent}{}{}
\patchcmd{\@setaddresses}{\indent}{\noindent}{}{}
\patchcmd{\@setaddresses}{\indent}{\noindent}{}{}
\patchcmd{\@setaddresses}{\indent}{\noindent}{}{}
\makeatother

\makeatletter
\DeclareMathSizes{12}{12}{6}{5}
\makeatother

\makeatletter
\def\namedlabel#1#2{\begingroup
    #2%
    \def\@currentlabel{#2}%
    \phantomsection\label{#1}\endgroup
}
\makeatother

\allowdisplaybreaks

\theoremstyle{plain}

\newtheorem{theorem}{Theorem}[section]
\newtheorem{lemma}[theorem]{Lemma}
\newtheorem{proposition}[theorem]{Proposition}
\newtheorem{corollary}[theorem]{Corollary}
\newtheorem*{theorem*}{Theorem}

\theoremstyle{definition}
\newtheorem{definition}[theorem]{Definition}
\newtheorem{example}[theorem]{Example}

\theoremstyle{remark}
\newtheorem{remark}[theorem]{Remark}

\newcommand{\K}{\Bbbk}

\newcommand{\op}[1]{#1^{o}} 
\newcommand{\env}[1]{{#1}^{e}} 
\newcommand{\mf}[1]{\mathfrak{#1}} 

\newcommand{\tensor}[1]{\otimes_{{#1}}} 

\newcommand{\id}{\mathsf{Id}} 

\newcommand{\Hom}[6]{{_{#1}^{#2}\mathsf{Hom}_{#3}^{#4}}\left({#5},{#6}\right)} 
\newcommand{\End}[2]{\mathsf{End}_{#1}(#2)} 
\newcommand{\Derk}[1]{\mathsf{Der}_{\K}\left({#1}\right)} 
\newcommand{\Der}{\mathsf{Der}}

\newcommand{\alg}{\mathsf{Alg}}
\newcommand{\algk}{\alg_{\K}} 
\newcommand{\liek}{\mathsf{Lie}_{\K}} 
\newcommand{\vectk}{\mathsf{Vect}_{\K}} 
\newcommand{\Bim}[2]{{}_{#1}\M{}_{#2}} 
\newcommand{\Bimod}[1]{\Bim{#1}{#1}} 
\newcommand{\ALie}[1]{\mathsf{AnchLie}_{#1}} 
\newcommand{\Ring}{\mathsf{Ring}} 
\newcommand{\LieRin}{\mathsf{LieRin}} 

\newcommand{\cA}{{\mathcal A}}

\newcommand{\cC}{{\mathcal C}}

\newcommand{\cE}{{\mathcal E}}
\newcommand{\cF}{{\mathcal F}}

\newcommand{\cL}{{\mathcal L}}
\newcommand{\cM}{{\mathcal M}}

\newcommand{\cU}{{\mathcal U}}

\newcommand{\sL}{\mathscr{L}}

\newcommand{\sU}{\mathscr{U}}

\newcommand{\M}{\mathsf{Mod}} 
\newcommand{\calpha}{\mf{a}}
\newcommand{\clambda}{\mf{l}}
\newcommand{\crho}{\mf{r}}

\newcommand{\eg}{e.g.~}

\newbox\pullbackbox
\setbox\pullbackbox=\hbox{\xy 0;<1mm,0mm>: \POS(2,0)\ar@{-} (0,2) \ar@{-} (4,2)
\endxy}

\newbox\pushoutbox
\setbox\pushoutbox=\hbox{\xy 0;<1mm,0mm>: \POS(2,2)\ar@{-} (0,0) \ar@{-} (4,0)
\endxy}

\definecolor{bostonuniversityred}{rgb}{0.8, 0.0, 0.0}

\title[UEAs of LRAs as left adjoint]{Universal enveloping algebras of Lie-Rinehart algebras as a left adjoint functor}

\author{Paolo Saracco}
\address{D\'epartement de Math\'ematique, Universit\'e Libre de Bruxelles, Boulevard du Triomphe, B-1050 Brussels, Belgium.}
\urladdr{sites.google.com/view/paolo-saracco}
\urladdr{homepages.ulb.ac.be/~psaracco}
\email{paolo.saracco@ulb.be}

\thanks{Paolo Saracco is a Charg\'e de Recherches of the Fonds de la Recherche Scientifique - FNRS and a member of the ``National Group for Algebraic and Geometric Structures and their Applications'' (GNSAGA-INdAM)}

\begin{document}

\begin{abstract}
We prove how the universal enveloping algebra constructions for Lie-Rinehart algebras and anchored Lie algebras are naturally left adjoint functors. This provides a conceptual motivation for the universal properties these constructions satisfy. As a supplement, the categorical approach offers new insights into the definitions of Lie-Rinehart algebra morphisms, of modules over Lie-Rinehart algebras and of the infinitesimal gauge algebra of a module.
\end{abstract}

\maketitle

\tableofcontents


\section*{Introduction}

It is well-known that any associative and unital algebra $A$ over a field $\K$ admits a natural Lie algebra structure with bracket given by the commutator bracket. Usually, one writes $\cL(A)$ for the Lie algebra associated with $A$ and the assignment $A \mapsto \cL(A)$ turns out to be functorial from the category of $\K$-algebras $\algk$ to the category of Lie algebras $\liek$. In the opposite direction, one defines the \emph{universal enveloping $\K$-algebra} of a given a Lie algebra $L$ over $\K$ to be an associative and unital algebra $U(L)$ together with a natural morphism of Lie algebras $L \to \cL(U(L))$ which is universal from $L$ to the functor $\cL$ (see, for instance, \cite[\S2.1]{Bourbaki}). Equivalently, the universal enveloping algebra construction provides a left adjoint $U:\liek \to \algk$ to the associated Lie algebra functor $\cL$.

Among Lie algebras, a family that particularly attracted the attention of the community is that of Lie-Rinehart algebras. Informally, a Lie-Rinehart algebra over a commutative algebra $R$ (also called $(\K,R)$-Lie algebra) is a Lie algebra $L$ together with an additional structure that reflects the interaction between the algebra of smooth functions on a smooth manifold $\cM$ and the Lie algebra of smooth vector fields on $\cM$. Rinehart himself gave an explicit construction of what he called the universal enveloping algebra $U(R, L)$ of a Lie-Rinehart algebra in \cite{Rinehart} and proved a Poincar\'e-Birkhoff-Witt theorem for the latter. Other equivalent constructions are provided in \cite[\S3.2]{LaiachiPaolo-complete}, \cite[page 64]{Huebschmann-Poisson}, \cite[\S18]{Sweedler-groups}. The universal property of $U(R, L)$ as an algebra is spelled out in \cite[page 64]{Huebschmann-Poisson} and \cite[page 174]{Malliavin} (where it is attributed to Feld'man). Its universal property as an $R$-bialgebroid is codified in the Cartier-Milnor-Moore Theorem for $U(R, L)$ proved in \cite[\S3]{MoerdijkLie}, where it is shown that the construction of the universal enveloping algebra provides a left adjoint to the functor sending any cocommutative bialgebroid to its Lie-Rinehart algebra of primitive elements. Further algebraic and categorical properties and applications are investigated in \cite{ArdiLaiachiPaolo,LaiachiPaolo-diff,Huebschmann-quantization,Huebschmann-LR,Huebschmann-Poisson}.

However, as far as the author is aware, nowhere it is shown that the construction of the so-called ``universal enveloping algebra'' provides a left adjoint functor. In fact, the universal property of $U(R,L)$ mentioned above involves at the same time a morphism of $\K$-algebras $R \to U(R,L)$ and a morphism of Lie algebras $L \to \cL\big(U(R,L)\big)$ which need to be compatible in a suitable way and there seems to be no natural functor from the category of $\K$-algebras (or that of $R$-rings) to the category of Lie-Rinehart algebras over $R$ which may play the role of the right adjoint functor.

Our aim is to make up for this lack by explicitly providing such a right adjoint. As a consequence, our results provide a conceptual motivation for addressing $U(R,L)$ as the universal enveloping $R$-ring of $L$. Notice also that being able to identify the universal enveloping algebra functor as a left adjoint offers a number of additional informations from the categorical point of view. For instance, that it preserves all colimits that exists in the category of Lie-Rinehart algebras over $R$. Furthermore, we will comment on how the categorical approach and the constructions we introduce allow us to re-interpret and explain the notions of modules over a Lie-Rinehart algebra $L$, of morphisms of Lie-Rinehart algebras over different bases and of the infinitesimal gauge algebra of an $A$-module $M$ with respect to a Lie-Rinehart algebra $L$, as introduced in \cite{Huebschmann-Poisson}. 

Even if our main objective concerns Lie-Rinehart algebras and their universal enveloping algebras, we will initially work in the more general framework of anchored Lie algebras as introduced in \cite[\S2]{Saracco-anch}. Our motivation for this choice is twofold. Firstly, the non-commutative setting of anchored Lie algebras is, at the same time, general enough to become easier to be handle and close enough to our target to provide most of the results we need in order to deal with the commutative setting. Secondly, anchored Lie algebras subsume several important examples of Lie algebras acting by derivations on associative algebras which are not necessarily commutative (any Lie algebra acts by derivations on its universal enveloping algebra and any associative algebra acts by inner derivations on itself) and they proved to be of key importance in the study of the structure of primitively generated bialgebroids over a non-commutative base (see \cite[\S4]{Saracco-anch}). Therefore, they deserve to be studied more closely. In addition, it is noteworthy that our approach via anchored Lie algebras allows also to re-interpret the algebra of differential operators of a representation of a Lie algebra, as introduced by Jacobson in \cite[page 175]{Jacobson}, and to provide a conceptual interpretation for its universal property (see discussion after Theorem \ref{thm:adj1}).

Concretely, after an introductory Section \ref{sec:prelim}, where we collect the basics on anchored Lie algebras and Lie-Rinehart algebras that we need to keep the presentation self-contained, in Section \ref{sec:UEAeRing} we explicitly construct a functor $\sL_A:\Ring_{\env{A}} \to \ALie{A}$ which we prove to be the natural right adjoint to the universal enveloping $\env{A}$-ring functor $\sU_A:\ALie{A} \to \Ring_{\env{A}}$ provided by the Connes-Moscovici's bialgebroid construction in \cite[\S2.2]{Saracco-anch} (see Theorem \ref{thm:main}). Section \ref{sec:LRalg} is devoted to adapt the results of Section \ref{sec:UEAeRing} to the commutative framework of Lie-Rinehart algebras and so to construct a right adjoint $\cL_A:\Ring_A \to \LieRin_A$ to the universal enveloping algebra functor $\cU_A:\LieRin_A \to \Ring_A$ (see Theorem \ref{thm:main2}). We conclude with a brief reflection in Section \ref{sec:applications} concerning the definitions of morphisms of Lie-Rinehart algebras over different bases, of modules over a Lie-Rinehart algebra and of the Lie-Rinehart algebra of infinitesimal gauge transformations of an $A$-module.

\subsection*{Conventions}

All over the paper, we assume a certain familiarity of the reader with the language of monoidal categories and of (co)monoids therein (see, for example, \cite[VII]{MacLane}).

We work over a ground field $\K$ of characteristic $0$. All vector spaces are assumed to be over $\K$. The unadorned tensor product $\otimes$ stands for $\tensor{\K}$. All (co)algebras and bialgebras are intended to be  $\K$-(co)algebras and $\K$-bialgebras, that is to say, (co)algebras and bialgebras in the symmetric monoidal category of vector spaces $\left(\vectk,\otimes ,\K\right) $. Every (co)module has an underlying vector space structure. Identity morphisms $\id_V$ are often denoted simply by $V$.

If $f : U \to V$ is a $\K$-linear map, we denote by $f^*$ the morphism $\Hom{}{}{\K}{}{V}{W} \to \Hom{}{}{\K}{}{U}{W}, g \mapsto g \circ f,$ and by $f_*$ the morphism $\Hom{}{}{\K}{}{W}{U} \to \Hom{}{}{\K}{}{W}{V}, g \mapsto f \circ g$.

Unless stated otherwise, $A$ denotes a not necessarily commutative algebra over $\K$ and $\op{A}$ denotes its opposite algebra. When looking at $a \in A$ as an element in $\op{A}$, we will denote it by $\op{a}$.
The Lie algebra associated with $A$, which by abuse of notation we denote by $A$ again, is the Lie algebra which has as underlying $\K$-vector space $A$ itself and bracket given by the commutator $[a,b] = ab-ba$ for all $a,b \in A$.

If $C$ is a $\K$-coalgebra, we take advantage of the Heyneman-Sweedler notation $\Delta(c) = \sum c_1 \otimes c_2$ to deal explicitly with the comultiplication.

Finally, we will make use of the fact that algebraic forgetful functors create limits in the sense of \cite[V.1, Definition]{MacLane}. In particular, the pullback of a diagram of vector spaces (or modules) can be computed by performing the corresponding pullback in the category of sets and by endowing it with the unique structure that makes of it a vector space (or module).


\section{Preliminaries}\label{sec:prelim}

We collect some facts about bimodules, rings, Lie-Rinehart algebras and anchored Lie algebras that will be needed in the sequel. The aim is to keep the exposition self-contained. 

\subsection{$A$-bimodules and $A$-rings} Given a $\K$-algebra $A$, the category of $A$-bimodules forms a non-strict monoidal category $\left(\Bimod{A},\tensor{A},A,\calpha,\clambda,\crho\right)$. 
Nevertheless, all over the paper we will behave as if the structural natural isomorphisms
\begin{gather*}
\calpha_{M,N,P}:\left(M\tensor{A}N\right)\tensor{A}P \to M\tensor{A}\left(N\tensor{A}P\right), \quad \left(m\tensor{A}n\right)\tensor{A}p \to m\tensor{A}\left(n\tensor{A}p\right), \\
\clambda_M: A \tensor{A} M \to M, \quad a\tensor{A}m \mapsto a\cdot m, \qquad \text{and} \qquad  \crho_M:M\tensor{A}A \to M, \quad m \tensor{A} a \mapsto m \cdot a,
\end{gather*}
were ``the identities'', that is, as if $\Bimod{A}$ was a strict monoidal category. 

An $A$\emph{-ring} is a monoid in $(\Bimod{A},\tensor{A},A)$. Equivalently, it is a $\K$-algebra $R$ together with a morphism of $\K$-algebras $\phi:A \to R$. In what follows we will often write that $(R,\phi)$ is an $A$-ring when we need to stress the role played by $\phi$. A \emph{morphism of $A$-rings} from $(R,\phi)$ to $(S,\psi)$ is a $\K$-algebra morphism $\varphi:R \to S$ such that $\varphi \circ \phi = \psi$. The category of $A$-rings and their morphisms will be denoted by $\Ring_A$.

\subsection{Lie-Rinehart algebras}\label{ssec:LRalg} A \emph{Lie-Rinehart algebra} over a commutative $\K$-algebra $A$ (called in this way in honour of Rinehart, who studied them in \cite{Rinehart} under the name of $(K,R)$-Lie algebras) is a Lie algebra $L$ endowed with a (left) $A$-module structure $A\otimes L\to L, a \otimes X \mapsto a \cdot X,$ and with a Lie algebra morphism $\omega:L\to \Derk{A} $ such that
\begin{equation}\label{eq:Leibniz}
\omega \left( a\cdot X\right) = a\cdot \omega \left( X\right) \qquad \text{and}\qquad \left[ X,a\cdot Y\right] = a \cdot \left[ X,Y\right] +\omega \left(X\right) \left( a\right) \cdot Y
\end{equation}
for all $a \in A$ and $X,Y\in L$. A \emph{morphism of Lie-Rinehart algebras} over $A$ from $\left( L,\omega \right) $ to $\left( L^{\prime },\omega ^{\prime }\right) $ is a Lie algebra morphism $f:L\to L^{\prime }$ which is also left $A$-linear and such that $\omega^{\prime }\circ f=\omega $. With these morphisms, Lie-Rinehart algebras over $A$ form a category that we denote by $\LieRin_A$. For the sake of brevity, we may simply write that $(A,L,\omega)$ is a Lie-Rinehart algebra to mean that $(L,\omega)$ is a Lie-Rinehart algebra over $A$. When $A$ and $\omega$ are clear from the context, we may simply write $L$ instead of $(A,L,\omega)$. For instance, we may often write that $f: L \to L'$ is a morphism of Lie-Rinehart algebras over $A$ to mean that $f$ is a morphism in $\LieRin_A$ as above.

\begin{example}
The smooth global sections of a Lie algebroid over a real smooth manifold $\cM$ naturally form a Lie-Rinehart algebra over $\cC(\cM)$. In particular, the smooth vector fields on $\cM$ give rise to the Lie-Rinehart algebra $\big(\Derk{\cC(\cM)},\id\big)$.
\end{example}

Given a Lie-Rinehart algebra $(A,L,\omega)$, a \emph{universal enveloping algebra} of $L$ is an $A$-ring $\big(\cU_A(L),\iota_A\big)$ together with a morphism of Lie algebras $\iota_L:L \to \cU_A(L)$ such that
\begin{equation}\label{eq:UEAm}
\iota_L(a\cdot X) = \iota_A(a)\iota_L(X) \qquad \text{and} \qquad \big[\iota_L(X),\iota_A(a)\big] = \iota_A\big(\omega(X)(a)\big)
\end{equation}
for all $a \in A$, $X \in L$, and which is universal with respect to this property. That is to say, for any $A$-ring $(R,\phi_A)$ together with a Lie algebra morphism $\phi_L:L \to R$ such that 
\begin{equation}\label{eq:UEA}
\phi_L(a\cdot X) = \phi_A(a)\phi_L(X) \qquad \text{and} \qquad \big[\phi_L(X),\phi_A(a)\big] = \phi_A\big(\omega(X)(a)\big)\
\end{equation}
for all $a \in A$, $X \in L$, there exists a unique morphism of $A$-rings $\Phi:\cU_A(L) \to R$ such that $\Phi \circ \iota_L = \phi_L$. It follows that the universal enveloping algebra construction induces a functor $\cU_A:\LieRin_A \to \Ring_A$.

\subsection{Anchored Lie algebras}\label{ssec:Aanch} Recall from \cite[\S2.1]{Saracco-anch} that an \emph{anchored Lie algebra} over a non-commutative $\K$-algebra $A$ (also called $A$-anchored Lie algebra) is a Lie algebra $L$ over $\K$ together with a Lie algebra morphism $\omega: L \to \Der_\K(A)$, called the \emph{anchor}. A \emph{morphism of $A$-anchored Lie algebras} from $\left( L,\omega \right) $ to $\left( L^{\prime },\omega ^{\prime }\right) $ is a Lie algebra morphism $f:L\to L^{\prime }$ such that $\omega^{\prime }\circ f=\omega $. The category of $A$-anchored Lie algebras and their morphisms will be denoted by $\ALie{A}$. As a matter of notation, we may write $(A,L,\omega)$ to mean the $A$-anchored Lie algebra $(L,\omega)$. Again, if $A$ and $\omega$ are clear from the context, we may simply write $L$ instead of $(A,L,\omega)$.

\begin{example}\label{ex:Ranch}
The associated Lie algebra of a $\K$-algebra $A$ becomes an $A$-anchored Lie algebra with anchor $\varpi_A : A \to \Der_\K(A), a \mapsto [a,-]$.
\end{example}

Given an $A$-anchored Lie algebra $(L,\omega)$, we have a construction for a \emph{universal enveloping algebra} of $L$. This is the $\env{A}$-ring $A \odot U(L) \odot A$ obtained by the Connes-Moscovici's bialgebroid construction in \cite[\S2.2]{Saracco-anch}. Explicitly, it is the $\K$-algebra with underlying vector space $A \otimes U(L) \otimes A$, unit $1_A \otimes 1_U \otimes 1_A$, multiplication uniquely determined by
\[(a \otimes u \otimes b)(a'\otimes v \otimes b') = \sum a(u_1 \cdot a') \otimes u_2v \otimes (u_3 \cdot b')b\]
and $\env{A}$-ring structure given by $J_A:\env{A} \to A \odot U(L) \odot A, a \otimes \op{b} \mapsto a \otimes 1_U \otimes b$.
It comes with a morphism of Lie algebras $J_L:L \to A \odot U(L) \odot A, X \mapsto 1_A \otimes X \otimes 1_A,$ such that
\[
\big[J_L(X),J_A(a \otimes \op{b})\big] = J_A\big(\omega(X)(a) \otimes \op{b} + a \otimes \op{\omega(X)(b)}\big)
\]
for all $a,b \in A$, $X \in L$, and which is universal with respect to this property. That is to say, for any $\env{A}$-ring $(R,\phi_A)$ together with a Lie algebra morphism $\phi_L:L \to R$ such that 
\[
\big[\phi_L(X),\phi_A(a \otimes \op{b})\big] = \phi_A\big(\omega(X)(a) \otimes \op{b} + a \otimes \op{\omega(X)(b)}\big)\
\]
for all $a,b \in A$, $X \in L$, there exists a unique morphism of $\env{A}$-rings $\Phi:A \odot U(L) \odot A \to R$ such that $\Phi \circ J_L = \phi_L$. It follows that this construction induces a functor
\begin{equation}\label{eq:UEAdF}
\ALie{A} \to \Ring_{\env{A}}, \qquad (L,\omega) \mapsto \big(A \odot U(L) \odot A,J_A\big).
\end{equation}
One of our main aims is to show that the universal property characterizing $A \odot U(L) \odot A$ is expressing the fact that the functor \eqref{eq:UEAdF} is a left adjoint functor. As a consequence, we will refer to $A \odot U(L) \odot A$ as the universal enveloping $\env{A}$-ring of $(A,L,\omega)$.

The following results clarify the relationship between $A$-anchored Lie algebras and $\env{A}$-anchored Lie algebras that will be needed in the sequel.

%

\begin{lemma}\label{lem:ders}
Any $\delta \in \Derk{A}$ induces a derivation $\delta_\otimes \in \Derk{\env{A}}$ uniquely determined by 
$ 
\delta_\otimes\big(a \otimes \op{b}\big) = \delta(a) \otimes \op{b} + a \otimes \op{\delta(b)}
$ 
for all $a,b\in A$. This induces a Lie algebra morphism $e:\Derk{A} \to \Derk{\env{A}},\delta\mapsto \delta_\otimes$.
\end{lemma}

\begin{proof}
Straightforward.
\end{proof}

\begin{proposition}\label{prop:E}
Any $A$-anchored Lie algebra $(L,\omega)$ is an $\env{A}$-anchored Lie algebra via
\[\omega_\otimes:L \to \Derk{\env{A}}, \qquad X \mapsto \omega(X)_\otimes.\]
The assignment $(L,\omega) \mapsto (L,\omega_\otimes)$ provides a functor $\cE_A:\ALie{A} \to \ALie{\env{A}}$.
\end{proposition}

\begin{proof}
In view of Lemma \ref{lem:ders}, the composition $L \xrightarrow{\omega} \Derk{A} \xrightarrow{e} \Derk{\env{A}}$ mapping $X$ to $\omega(X)_\otimes$ is a Lie algebra map. 
Thus, the assignment $(L,\omega) \mapsto (L,\omega_\otimes)$ is well-defined and it is clearly functorial (it acts as the identity on morphisms). 
\end{proof}

The functor $\cE_A$ of Proposition \ref{prop:E} naturally admits a left adjoint functor. In order to introduce it, we take advantage of the construction of the product in the category $\ALie{A}$.

\begin{proposition}\label{prop:prod}
Let $(L,\omega)$ and $(L',\omega')$ be $A$-anchored Lie algebras. The product of $(L,\omega)$ and $(L',\omega')$ in $\ALie{A}$ exists and can be computed as the pullback of vector spaces
\begin{equation}\label{eq:prod}
\begin{gathered}
\xymatrix @=17pt{
L \underset{\Derk{A}}{\times} L' \ar[r]^-{q_1} \ar[d]_-{q_2} \ar@{}[dr]|(0.3){\Big\lrcorner} & L \ar[d]^-{\omega} \\
L' \ar[r]_-{\omega'} & \Derk{A},
}
\end{gathered}
\end{equation}
with component-wise bracket and anchor $\omega_\times \coloneqq \omega\circ q_1 = \omega'\circ q_2$.
\end{proposition}

\begin{proof}
Since the pullback in vector spaces of a diagram of Lie algebras is naturally a Lie algebra with component-wise bracket, $L \times_{\Derk{A}} L'$ is a Lie algebra and $q_1,q_2$ are Lie algebra maps. The fact that the $A$-anchored Lie algebra $\big(L \times_{\Derk{A}} L',\omega_\times\big)$ satisfies the universal property of the product is an easy check that we leave to the interested reader.
\end{proof}

Let $(M,\varpi)$ be an $\env{A}$-anchored Lie algebra. By Lemma \ref{lem:ders}, $\big(\Derk{A},e\big)$ is an $\env{A}$-anchored Lie algebra as well and we can consider their product in $\ALie{\env{A}}$:
\[
\begin{gathered}
\xymatrix @=17pt{
\cF_A(M,\varpi) \ar[r]^-{q_1} \ar[d]_-{q_2} \ar@{}[dr]|(0.25){\Big\lrcorner} & M \ar[d]^-{\varpi} \\
\Derk{A} \ar[r]_-{e} & \Derk{\env{A}}
}
\end{gathered}
\]
with component-wise bracket and anchor $e \circ q_2 = \varpi \circ q_1$. Concretely,
\[
\cF_A(M,\varpi) \hspace{-1pt} = \hspace{-1pt} \Big\{ \hspace{-1pt}\left(X,\delta\right) \hspace{-1pt} \in \hspace{-1pt} M \hspace{-1pt} \times \hspace{-1pt} \Der_\K(A) \,\big|\, \varpi(X)(a \hspace{-1pt} \otimes \hspace{-1pt} \op{b}) = \delta(a) \hspace{-1pt} \otimes \hspace{-1pt} \op{b} \hspace{-1pt} + \hspace{-1pt} a \hspace{-1pt} \otimes \hspace{-1pt} \op{\delta(b)}, \forall a,b \in A \Big\}.
\]

\begin{theorem}\label{thm:AeAadj}
For every $\env{A}$-anchored Lie algebra $(M,\varpi)$, the $\K$-vector space $\cF_A(M,\varpi)$ is an $A$-anchored Lie algebra with component-wise bracket and anchor $q_2$. The assignment $(M,\varpi) \mapsto \big(\cF_A(M,\varpi),q_2\big)$ induces a functor $\cF_A : \ALie{\env{A}} \to \ALie{A}$ which is right adjoint to the functor $\cE_A:\ALie{A} \to \ALie{\env{A}}$.
\end{theorem}

\begin{proof}
We already know that $\cF_A(M,\varpi)$ is a Lie algebra and that $q_2$ is a morphism of Lie algebras.
Moreover, the assignment $(M,\varpi) \mapsto \big(\cF_A(M,\varpi),q_2\big)$ is functorial.

To show that $\cF_A$ is right adjoint to $\cE_A:\ALie{A} \to \ALie{\env{A}}$, notice that $L$ itself with $q_1 = \id$ and $q_2 = \omega$ is a pullback of the pair $(e,\omega_\otimes)$. 
Therefore, the identity morphism of $L$ plays the role of the component of the unit at $L$ and it is a universal arrow from $L$ to $\cF_A$.
%
In view of \cite[III.1, Theorem 2(i)]{MacLane}, the pair $\cE_A,\cF_A$ forms an adjoint pair of functors.
\end{proof}

Summing up, for every $\K$-algebra $A$ we have an adjunction
\begin{equation}\label{eq:AeAadj}
\begin{gathered}
\xymatrix@=17pt{
\ALie{\env{A}} \ar@/^2ex/@<+0.5ex>[d]^-{\cF_{A}} \\
\ALie{A}. \ar@/^2ex/@<+0.5ex>[u]^-{\cE_{A}}
}
\end{gathered}
\end{equation}


\section{The universal enveloping $\env{A}$-ring as a left adjoint functor}\label{sec:UEAeRing}

Let $A,R$ be $\K$-algebras and let $(M,\varpi)$ be an $R$-anchored Lie algebra. Assume that a morphism of $\K$-algebras $\phi:A \to R$ has been given. In this way, $R$ becomes an $A$-ring and we can consider the $\K$-vector space
\[
\Der_{\K}(A,R) = \Big\{f \in \Hom{}{}{\K}{}{A}{R} ~\big|~ f(ab) = f(a)\phi(b) + \phi(a)f(b) \text{ for all }a,b \in A\Big\}.
\]
Define $\cA_A^R(M,\varpi)$ to be the pullback
\begin{equation}\label{eq:diagram}
\begin{gathered}
\xymatrix@=13pt{
\cA_A^R(M,\varpi) \ar@{}[ddr]|(0.2){\Big\lrcorner} \ar[r]^-{p_1} \ar[dd]_-{p_2} & M \ar[d]^-{\varpi} \\
 & \Der_\K(R) \ar[d]^-{\phi^*} \\
\Der_\K(A) \ar[r]_-{\phi_*} & \Der_\K(A,R)
}
\end{gathered}
\end{equation}
computed in the category of $\K$-vector spaces. Concretely,
\begin{equation}\label{eq:cL}
\cA_A^R(M,\varpi) = \Big\{\left(X,\delta\right) \in M \times \Der_\K(A) ~\big|~ \varpi(X)\big(\phi(a)\big) = \phi\big(\delta(a)\big) \text{ for all }a \in A\Big\}.
\end{equation}

\begin{lemma}\label{lem:keyAnch}
Let $(R,\phi)$ be an $A$-ring and $(M,\varpi)$ be an $R$-anchored Lie algebra. Then the $\K$-vector space $\cA_A^R(M,\varpi)$ is an $A$-anchored Lie algebra with anchor $p_2:\cA_A^R(M,\varpi) \to \Der_\K(A)$ and with component-wise bracket. Furthermore, $p_1$ is a Lie algebra morphism.
\end{lemma}

\begin{proof}
Let $(X,\delta), (X',\delta') \in \cA_A^R(M ,\varpi)$. By using \eqref{eq:cL} one checks that
\[
\varpi\big([X,X']\big)\big(\phi(a)\big) = \phi\Big([\delta,\delta'](a)\Big)
\]
for all $a \in A$. Therefore $\big([X,X'],[\delta,\delta']\big) \in \cA_A^R(M ,\varpi)$, which shows that $\cA_A^R(M ,\varpi)$ is a Lie algebra. Clearly, $p_2$ is a morphism of Lie algebras with respect to this structure and hence $\cA_A^R(M ,\varpi)$ is an $A$-anchored Lie algebra, as claimed.
\end{proof}

\begin{proposition}\label{prop:keyAnch}
For any $A$-ring $(R,\phi)$, the assignment $(M,\varpi) \mapsto \cA_A^R(M,\varpi)$ induces a functor $\cA_A^R:\ALie{R} \to \ALie{A}$.
\end{proposition}

\begin{proof}
We already know how $\cA_A^R$ acts on objects. To see how it acts on morphism, notice that it $f : (M,\varpi) \to (M',\varpi')$ is a morphism of $R$-anchored Lie algebras then the external hexagon in the following diagram
\[
\xymatrix @=14pt{
\cA_A^R\big(M,\varpi\big) \ar@{.>}[dr]_-{\cA_A^R(f)} \ar[r]^-{p_1} \ar@/_3ex/[dddr]_-{p_2} & M \ar[dr]^-{f} & \\
 & \cA_A^R\big(M' ,\varpi'\big) \ar@{}[ddr]|(0.2){\Big\lrcorner} \ar[r]_-{p_1'} \ar[dd]_-{p_2'} & M' \ar[d]^-{\varpi'} \\
 & & \Der_\K(R) \ar[d]^-{\phi^*} \\
 & \Der_\K(A) \ar[r]_-{\phi_*} & \Der_\K(A,R),
}
\]
commutes by definition of $\cA_A^R(M,\varpi)$. Thus, by the universal property of the pullback, there exists a unique morphism of $\K$-vector spaces $\cA_A^R(f):\cA_A^R\big(M ,\varpi\big) \to \cA_A^R\big(M',\varpi'\big)$ such that $p_2'\circ \cA_A^R(f) = p_2$ and $p_1'\circ \cA_A^R(f) = f \circ p_1$, which is explicitly given by
\begin{equation}\label{eq:indmapanch}
\cA_A^R(f) : \big(X,\delta\big) \mapsto \big(f(X),\delta\big).
\end{equation}
Since $f$ is a morphism of Lie algebras, it follows that $\cA_A^R(f)$ is of Lie algebras as well. The compatibility with the anchors is clear, whence $\cA_A^R(f)$ is of $A$-anchored Lie algebras. Since on arrows $\cA_A^R$ is defined in terms of a universal property, it is functorial. 
\end{proof}


\begin{corollary}\label{cor:RAanch}
Let $(R,\phi)$ be an $A$-ring and denote by $\varpi_R: R \to \Der_\K(R)$ the structure of $R$-anchored Lie algebra induced on $R$ by the commutator bracket as in Example \ref{ex:Ranch}. The $\K$-vector space $\cA_A^R(R,\varpi_R)$ is an $A$-anchored Lie algebra with anchor $p_2:\cA_A^R(R,\varpi_R) \to \Der_\K(A)$ and with component-wise bracket.
\end{corollary}



It follows from Corollary \ref{cor:RAanch} that to any $A$-ring $R$ we may assign an $A$-anchored Lie algebra $\big(\cA_A^R(R ,\varpi_R),p_2\big)$.

\begin{theorem}\label{thm:cLfuncts}
The assignment $R \mapsto \cA_A^R(R ,\varpi_R)$ induces a functor $\cL_A:\Ring_A \to \ALie{A}$.
\end{theorem}

\begin{proof}
We already know how $\cL_A$ acts on objects. To see how it acts on morphisms, notice that it $\varphi : (R,\phi) \to (S,\psi)$ is a morphism of $A$-rings then the right-most square and the lowest triangle in the following diagram commute:
\[
\xymatrix@=14pt{
\cA_A^R\big(R ,\varpi_R\big) \ar@{}[ddr]|(0.2){\Big\lrcorner} \ar[r]^-{p_1} \ar[dd]_-{p_2} & R  \ar[d]^-{\varpi_R} \ar[dr]^-{\varphi} & \\
 & \Der_\K(R) \ar[d]^-{\phi^*} & S \ar[d]^-{\varpi_S} \\
\Der_\K(A) \ar[r]_-{\phi_*} \ar@/_4ex/[drr]_-{\psi_*} & \Der_\K(A,R) \ar[dr]^-{\varphi_*} & \Der_\K(S) \ar[d]^-{\psi^*} \\
 & & \Der_\K(A,S).
}
\]
Therefore, by the universal property of the pullback, there exists a unique morphism of $\K$-vector spaces $\cL_A(\varphi):\cA_A^R\big(R ,\varpi_R\big) \to \cA_A^S\big(S ,\varpi_S\big)$, which is explicitly given by
\begin{equation}\label{eq:indmap}
\cL_A(\varphi) : (r,\delta) \mapsto \big(\varphi(r),\delta\big).
\end{equation}
Since $\varphi$ is a $\K$-algebra morphism satisfying $\varphi\circ \phi = \psi$, it follows that $\cL_A(\varphi)$ is a Lie algebra morphism. The compatibility with the anchors is clear, whence $\cL_A(\varphi)$ is of $A$-anchored Lie algebras. Since on arrows $\cL_A$ is defined in terms of a universal property, it is functorial and so we have a well-defined functor $\cL_A:\Ring_A \to \ALie{A}$.
\end{proof}

\begin{remark}
Let $\varphi:(R,\phi) \to (S,\psi)$ be a morphism of $A$-rings, as in the proof of Theorem \ref{thm:cLfuncts}. Set $\cA_R^S:\ALie{S} \to \ALie{R}$, $\cA_A^R:\ALie{R} \to \ALie{A}$ and $\cA_A^S:\ALie{S} \to \ALie{A}$ for the functors induced by $\varphi$, $\phi$ and $\psi$ respectively. The universal property of the pullback gives a morphism of $R$-anchored Lie algebras $\tilde{\varphi}:(R,\varpi_R) \to \cA_R^S(S,\varpi_S)$ induced by $\varphi$ and a morphism of $A$-anchored Lie algebras $\chi:\cA_A^R\big(\cA_R^S(S,\varpi_S)\big) \to \cA_A^S(S,\varpi_S)$ induced by the composition $\cA_A^R\big(\cA_R^S(S,\varpi_S)\big) \to \cA_R^S(S,\varpi_S) \to S$. The interested reader may check that the morphism $\cL_A(\varphi)$ coincides with the composition $\chi\circ\cA_A^R\big(\tilde{\varphi}\big)$. However, we believe that the elementary proof we gave of Theorem \ref{thm:cLfuncts} is more straightforward.
\end{remark}

Now assume that an $A$-anchored Lie algebra $(L,\omega)$ has been given. By the universal property of $U(L)$ there exists a unique $\K$-algebra extension $\Omega:U(L) \to \End{\K}{A}$ of $\omega$ which makes of $A$ a $U(L)$-module algebra with
\begin{equation}\label{eq:Umodalg}
X \cdot a = \omega(X)(a), \qquad X \cdot 1_A = 0 \qquad \text{and} \qquad u \cdot a = \Omega(u)(a)
\end{equation}
for all $X \in L$, $u \in U(L)$ and $a \in A$ (see \cite[Example 6.1.13(3)]{dascalescu}, for instance). As a consequence, we may consider the $A$-ring $A~\#~U(L)$ with underlying vector space $A \otimes U(L)$, unit $1_A \otimes 1_U$, multiplication uniquely determined by
\begin{equation}\label{eq:smash}
(a \otimes u)(b \otimes v) = \sum a(u_1\cdot b) \otimes u_2v
\end{equation}
for all $a,b \in A$, $u,v \in U(L)$, and $A$-ring structure $j_A:A \to A~\#~U(L), a \mapsto a \otimes 1_U$.

\begin{theorem}\label{thm:adj1}
The assignment $(L,\omega) \mapsto A~\#~U(L)$ induces a functor $\cU_A:\ALie{A} \to \Ring_A$ which is left adjoint to $\cL_A: \Ring_A \to \ALie{A}$.
\end{theorem}

\begin{proof}
Let $(L,\omega)$ be an $A$-anchored Lie algebra and set $\cU_A(L) \coloneqq A~\#~U(L)$. Consider the assignment 
$ 
j_L:L \to \cU_A(L), X \mapsto 1_A \otimes X.
$ 
In view of \eqref{eq:Umodalg} and of \eqref{eq:smash}, $j_L$ is a morphism of Lie algebras and
\[\big[j_L(X),j_A(a)\big] = \big[1_A \otimes X,a \otimes 1_U\big] = \omega(X)(a) \otimes 1_U = j_A\big(\omega(X)(a)\big)\]
 in $\cU_A(L)$, for all $X \in L$ and $a \in A$.
Therefore, $j_L$ and $\omega$ induce a $\K$-linear morphism
\begin{equation}\label{eq:unitAnch}
\eta_L:L \to \cL_A\big(\cU_A(L)\big), \qquad X \mapsto \big(\iota_L(X),\omega(X)\big),
\end{equation}
via the universal property of the pullback. It is easy to check that $\eta_L$ is of Lie algebras and that it is compatible with the anchors, thus it is a morphism of $A$-anchored Lie algebras. 
We claim that $\eta_L$ is a universal map from $L$ to $\cL_A$ in the sense of \cite[III.1, Definition]{MacLane}.
Assume then that $(R,\phi)$ is a $A$-ring and that $f:L \to \cL_A(R)$ is a morphism of $A$-anchored Lie algebras. By definition of $\cL_A(R)$, $f(X) = \big(\tilde{f}(X),\omega(X)\big)$ and 
\begin{equation}\label{eq:multiplicative}
\big[\tilde{f}(X),\phi(a)\big] = \varpi_R\big(\tilde{f}(X)\big)\big(\phi(a)\big) \stackrel{\eqref{eq:cL}}{=} \phi\big(\omega(X)(a)\big)
\end{equation}
for all $X \in L$ and $a \in A$, where $\tilde{f} \coloneqq p_1 \circ f$.
As in the proof of \cite[Theorem 2.9]{Saracco-anch}, a straightforward check using \eqref{eq:multiplicative} and induction on a PBW basis of $U(L)$ shows that
\[F:A~\#~U(L) \to R, \qquad a \otimes u \mapsto \phi(a)U\big(\tilde{f}\big)(u),\]
is a morphism of $A$-rings which satisfies $F\big(\iota_L(X)\big) = \tilde{f}(X)$ for all $X \in L$.
Moreover,
\[\cL_A(F)\big(\eta_L(X)\big) \stackrel{\eqref{eq:unitAnch}}{=} \cL_A(F)\big(\iota_L(X),\omega(X)\big) \stackrel{\eqref{eq:indmapanch}}{=} \Big(F\big(\iota_L(X)\big),\omega(X)\Big) = \big(\tilde{f}(X),\omega(X)\big) = f(X)\]
for all $X \in L$ and $F$ is the unique $A$-ring map satisfying the latter relation. 
Therefore, $\cU_A,\cL_A$ form an adjoint pair by \cite[IV.1, Theorem 2(ii)]{MacLane}.
\end{proof}

Notice that Theorem \ref{thm:adj1} is expressing the fact that $\big(A ~\#~ U(L),j_A\big)$ is the universal enveloping $A$-ring of $(A,L,\omega)$. In \cite[page 175]{Jacobson}, the algebra $A ~\#~ U(L)$ is called the \emph{algebra of differential operators} of the representation $\omega$ of $L$ (our construction differs slightly from the one in \cite{Jacobson}, because of the different choice of sides for the modules). Theorem \ref{thm:adj1} provides then a conceptual explanation for the universal property of $A ~\#~ U(L)$ described in \cite[V.6, Proposition 2]{Jacobson} and a new proof of the latter.

\begin{remark}
For the sake of future reference, if $f : (L,\omega) \to (L',\omega')$ is a morphism of $A$-anchored Lie algebras, then the induced morphism of $A$-rings $\cU_A(f):\cU_A(L) \to \cU_A(L')$ is explicitly given by
$ 
\cU_A(f)(a \otimes u) = a \otimes U(f)(u)
$ 
for all $a \in A$ and $u \in U(L)$.
\end{remark}

As a particular case of Theorem \ref{thm:adj1}, we have an adjunction
\[\xymatrix@=17pt{
\Ring_{\env{A}} \ar@/^2ex/@<+0.5ex>[d]^-{\cL_{\env{A}}} \\
\ALie{\env{A}}. \ar@/^2ex/@<+0.5ex>[u]^-{\cU_{\env{A}}}
}\]
If we compose it with the adjunction \eqref{eq:AeAadj}, we obtain a new adjunction
\begin{equation}\label{eq:CMadj}
\begin{gathered}
\xymatrix@=17pt{
\Ring_{\env{A}} \ar@/^2ex/@<+0.5ex>[d]^-{\sL_{A}} \\
\ALie{A}. \ar@/^2ex/@<+0.5ex>[u]^-{\sU_{A}}
}
\end{gathered}
\end{equation}
where $\sU_A \coloneqq \cU_{\env{A}} \circ \cE_{A}$ and $\sL_A \coloneqq \cF_A \circ \cL_{\env{A}}$. Notice that $\sU_A(L) = \env{A} ~\#~ U(L)$, where the $U(L)$-module structure on $\env{A}$ is that of a tensor product of $U(L)$-modules.

\begin{theorem}\label{thm:main}
There is an isomorphism of $\env{A}$-rings
\begin{equation}\label{eq:iso}
\env{A} ~\#~ U(L) \to A \odot U(L) \odot A, \qquad (a \otimes \op{b}) \otimes u \mapsto a \otimes u \otimes b.
\end{equation}
In particular, \eqref{eq:CMadj} exhibits the Connes-Moscovici's bialgebroid construction of \cite[\S2]{Saracco-anch} as a left adjoint functor. 
\end{theorem}

\begin{proof}
A straightforward computation by means of the cocommutativity of $U(L)$ shows that \eqref{eq:iso} is, in fact, a morphism of $\env{A}$-rings, where the $U(L)$-module structure on $\env{A}$ is given by the diagonal action $u \cdot (a \otimes \op{b}) = \sum (u_1\cdot a) \otimes \op{(u_2 \cdot b)}$ for all $u \in U(L)$, $a,b\in A$. 
\end{proof}

The universal property of \cite[Theorem 2.9]{Saracco-anch} (see \S\ref{ssec:Aanch}) expresses exactly the fact that for any morphism of $A$-anchored Lie algebras $L \to \sL_A(R)$, there exists a unique morphism of $\env{A}$-rings $A \odot U(L) \odot A \to R$ extending it, as the following proposition states.

\begin{proposition}\label{prop:univprop}
For any $\env{A}$-ring $(R,\phi_A)$, the $A$-anchored Lie algebra $\sL_A(R)$ can be realized as the following pullback of $\K$-vector spaces
\[
\xymatrix@=14pt{
\sL_A(R) \ar@{}[ddr]|(0.2){\Big\lrcorner} \ar[rr]^-{\rho_1} \ar[dd]_-{\rho_2} & & R \ar[d]^-{\varpi_R} \\
 & & \Der_\K(R) \ar[d]^-{{\phi_A}^*} \\
\Der_\K(A) \ar[r]_-{e} & \Derk{\env{A}} \ar[r]_-{{\phi_A}_*} & \Der_\K(\env{A},R)
}
\]
with component-wise bracket and anchor $\rho_2$. Concretely,
\[
\sL_A(R) = \Big\{\left(r,\delta\right) \in R \times \Derk{A} ~\big|~ \big[r,\phi_A(a \otimes \op{b})\big] = \phi_A\big(\delta(a) \otimes \op{b} + a \otimes \op{\delta(b)}\big), \forall a,b \in A\Big\}.
\]
The datum of a morphism of $A$-anchored Lie algebras $L \to \sL_A(R)$ is therefore equivalent to the datum of a morphism of Lie algebras $\phi_L:L \to R$ such that for all $X \in L$, $a,b \in A$,
\[\big[\phi_L(X),\phi_A(a \otimes \op{b})\big] = \phi_A\big(X\cdot(a \otimes \op{b})\big).\]
\end{proposition}

\begin{proof}
The first claim follows from the pasting law for pullbacks. The second claim is a straightforward check.
\end{proof}

If we define an $(A,L,\omega)$-module to be an $A$-bimodule together with a Lie algebra morphism $\rho: L \to \End{\K}{M}$ such that
\begin{equation}\label{eq:rho}
\rho(X)(a\cdot m \cdot b) = \omega(X)(a) \cdot m\cdot b + a\cdot \rho(X)(m) \cdot b + a \cdot m \cdot \omega(X)(b)
\end{equation}
as in \cite[Corollary 2.10]{Saracco-anch}, then we have the following expected result.

\begin{proposition}\label{prop:modsanch}
For an $A$-bimodule $M$, the datum of a left $(A,L,\omega)$-module structure is equivalent to the datum of a morphism $L \to \sL_A\big(\End{\K}{M}\big)$ of $A$-anchored Lie algebras.
\end{proposition}

\begin{proof}
Recall that if $M$ is an $A$-bimodule, then $\End{\K}{M}$ has a natural $\env{A}$-ring structure induced by left and right multiplication by $A$:
\[\phi:\env{A} \to \End{\K}{M}, \qquad a \otimes \op{b}\mapsto l_a \circ r_b,\]
where $l_a(m) \coloneqq a\cdot m$ and $r_a(m) \coloneqq m \cdot a$ for all $a \in A$, $m \in M$.
By the universal property of the pullback, giving a morphism $\varrho : L \to \sL_A\big(\End{\K}{M}\big)$ of $A$-anchored Lie algebras is equivalent to giving a morphism of Lie algebras $\rho :L \to \End{\K}{M}$ such that $\phi^* \circ \varpi_{\End{\K}{M}} \circ \rho = \phi_* \circ e \circ \omega$,
which is exactly \eqref{eq:rho}.
\end{proof}

As a consequence, Theorem \ref{thm:main} and Proposition \ref{prop:modsanch} provide a conceptual proof of the equivalence between the category of $(A,L,\omega)$-modules and the category of $\sU_A(L)$-modules already observed in \cite[Corollary 2.10]{Saracco-anch}.




\section{The universal enveloping $A$-ring as a left adjoint functor}\label{sec:LRalg}

Henceforth, $A$ is a commutative $\K$-algebra. Notice that if a morphism of $\K$-algebras $\phi:A \to R$ has been given, then $\Der_{\K}(A,R)$ becomes a left $A$-module with $A$-action $(a\cdot f)(b) \coloneqq \phi(a)f(b)$ for all $a,b \in A$, $f \in \Der_{\K}(A,R)$. 

\begin{proposition}\label{prop:key}
Let $(R,\phi)$ be an $A$-ring. The $\K$-vector space $\cA_A^R(R,\varpi_R)$ of \eqref{eq:cL} is a Lie-Rinehart algebra over $A$ with anchor $p_2:\cA_A^R(R,\varpi_R) \to \Der_\K(A)$ and with component-wise bracket and left $A$-action. Furthermore, $p_1$ is a left $A$-linear and Lie algebra morphism.
\end{proposition}

\begin{proof}
We already know from Lemma \ref{lem:keyAnch} and Corollary \ref{cor:RAanch} that $\cA_A^R(R,\varpi_R)$ is an $A$-anchored Lie algebra with component-wise bracket and anchor $p_2$ and we know that $p_1$ is a Lie algebra morphism. We only need to check the $A$-module properties. Since, in this case, \eqref{eq:diagram} is also a diagram of left $A$-modules and left $A$-linear morphisms, $\cA_A^R(R,\varpi_R)$ is a left $A$-module itself with component-wise $A$-action and $p_1,p_2$ are left $A$-linear.
Moreover
\begin{align*}
\Big[(r,\delta),a\cdot(r',\delta')\Big] & \stackrel{\phantom{(11)}}{=} \Big[(r,\delta),(\phi(a)r',a\cdot \delta')\Big] = \Big(r\phi(a)r'- \phi(a)r'r, \big[\delta,a\cdot \delta'\big]\Big) \\
 & \stackrel{\eqref{eq:cL}}{=} \Big(\phi(a)[r,r'] + \phi\big(\delta(a)\big)r', a\cdot\big[\delta,\delta'\big] + \delta(a)\cdot \delta'\Big) \\
 & \stackrel{\phantom{(11)}}{=} a\cdot \big([r,r'],[\delta,\delta']\big) + \delta(a)\cdot \big(r', \delta'\big) \\
 & \stackrel{\phantom{(11)}}{=} a \cdot \big[(r,\delta),(r',\delta')\big] + p_2\big((r,\delta)\big)(a)\cdot (r',\delta')
\end{align*}
for all $a \in A$, $r,r'\in R$, $ \delta,\delta'\in \Derk{A}$, which entails that the Leibniz rule \eqref{eq:Leibniz} is satisfied and hence that $\cA_A^R(R ,\varpi_R)$ is a Lie-Rinehart algebra over $A$.
\end{proof}

\begin{theorem}\label{thm:cLfunctor}
The assignment $R \mapsto \cA_A^R(R ,\varpi_R)$ induces a functor $\cL_A:\Ring_A \to \LieRin_A$.
\end{theorem}

\begin{proof}
It follows from Theorem \ref{thm:cLfuncts} and Proposition \ref{prop:key}.
\end{proof}

Concretely,
$
\cL_A(R) = \Big\{\left(r,\delta\right) \in R \times \Derk{A} ~\big|~ \big[r,\phi_A(a)\big] = \phi_A\big(\delta(a)\big) \text{ for all } a \in A\Big\}.
$


The following proposition, analogue of Proposition \ref{prop:univprop}, argues in favour of the fact that $\cL_A$ provides a right adjoint for the functor $\cU_A : \LieRin_A \to \Ring_A$ of \S\ref{ssec:LRalg}.

\begin{proposition}\label{prop:adj}
Let $(R,\phi_A)$ be an $A$-ring and let $(L,\omega)$ be a Lie-Rinehart algebra over $A$. Then the datum of a morphism $\psi_L:L \to \cL_A(R)$ of Lie-Rinehart algebras over $A$ is equivalent to the datum of a morphism of Lie algebras $\phi_L:L \to R $ such that \eqref{eq:UEA} hold.
\end{proposition}

\begin{proof}
By the universal property of the pullback, the existence of a morphism of $A$-modules and of Lie algebras $\psi_L:L \to \cL_A(R)$ such that $p_2\circ \psi_L = \omega$ is equivalent to the existence of a morphism of $A$-modules and of Lie algebras $\phi_L : L \to R$ such that ${\phi_A}^* \circ \varpi_R \circ \phi_L = {\phi_A}_* \circ \omega$, which is exactly \eqref{eq:UEA}.
\end{proof}


\begin{theorem}\label{thm:main2}
The functor $\cL_A:\Ring_A \to \LieRin_A$ is right adjoint to the universal enveloping algebra functor $\cU_A:\LieRin_A \to \Ring_A$.
\end{theorem}

\begin{proof}
Let $(A,L,\omega)$ be a Lie-Rinehart algebra and consider the assignment
\begin{equation}\label{eq:unit}
\eta_L:L \to \cL_A\big(\cU_A(L)\big), \qquad X \mapsto \big(\iota_L(X),\omega(X)\big),
\end{equation}
induced by $\iota_L$ and $\omega$ via the universal property of the pullback in view of \eqref{eq:UEAm}. 
It is easy to check that $\eta_L$ is a morphism of Lie-Rinehart algebras over $A$.
Moreover, if $f:(L,\omega) \to (L',\omega')$ is a morphism of Lie-Rinehart algebras over $A$ and if $F:\cL_A\big(\cU_A(L)\big) \to \cL_A\big(\cU_A(L')\big)$ denotes the morphism induced by $\cU_A(f)$, then the fact that
\[F\big(\eta_L(X)\big) \stackrel{\eqref{eq:unit}}{=} F\big(\iota_L(X),\omega(X)\big) \stackrel{\eqref{eq:indmap}}{=} \Big(\cU_A(f)\big(\iota_L(X)\big),\omega(X)\Big) = \big(\iota_{L'}\big(f(X)\big),\omega'\big(f(X)\big)\big)\]
for all $X \in L$ entails that $F \circ \eta_L = \eta_{L'} \circ f$ and so the collection $\left\{\eta_L\mid L \in \LieRin_A\right\}$ defines a natural transformation $\eta: \id \to \cL_A\circ \cU_A$.

In view of Proposition \ref{prop:adj} and the universal property of $\cU_A(L)$, for any $A$-ring $(R,\phi_A)$ and any morphism $\psi_L:L \to \cL_A(R)$ of Lie-Rinehart algebras over $A$, there exists a unique morphism of $A$-rings $\Phi:\cU_A(L) \to R$ such that $\Phi\circ \iota_L = p_1 \circ \psi_L$. By a direct check
\[\cL_A(\Phi)\big(\eta_L(X)\big) \stackrel{\eqref{eq:unit}}{=} \cL_A(\Phi)\big(\iota_L(X),\omega(X)\big) \stackrel{\eqref{eq:indmap}}{=} \Big(\Phi\big(\iota_L(X)\big),\omega(X)\Big) = \psi_L(X)\]
for all $X \in L$ and $\Phi$ is the unique morphism of $A$-rings satisfying $\cL_A(\Phi) \circ \eta_L = \psi_L$. 
It follows that $\eta_L$ is a universal map from $L$ to $\cL_A$ in the sense of \cite[III.1, Definition]{MacLane}, for every $(L,\omega) \in \LieRin_A$, and hence $\cU_A,\cL_A$ form an adjoint pair by \cite[IV.1, Theorem 2(i)]{MacLane}.
\end{proof}


\section{On morphisms, modules and the infinitesimal gauge algebra}\label{sec:applications}

We conclude with a few remarks concerning morphisms between Lie-Rinehart algebras over different bases, modules over Lie-Rinehart algebras and the infinitesimal gauge algebra $\mathrm{DO}(A,L,M)$ of an $A$-module $M$ with respect to $(A,L,\omega)$ described in \cite[page 72]{Huebschmann-Poisson}.


\subsection{Morphisms over different bases}
By mimicking the arguments used to prove Lemma \ref{lem:keyAnch} and Proposition \ref{prop:key}, one shows that the following result holds.

\begin{proposition}\label{prop:morph}
If $(L',\omega')$ is a Lie-Rinehart algebra over a commutative $\K$-algebra $A'$ and $\phi: A \to A'$ is a morphism of commutative $\K$-algebras, then $\cA_A^{A'}(L',\omega')$ is a Lie-Rinehart algebra over $A$.
\end{proposition}

Proposition \ref{prop:morph} suggests the following definition.

\begin{definition}\label{def:morph}
A morphism of Lie-Rinehart algebras from $(A,L,\omega)$ to $(A',L',\omega')$ is a pair $(\phi,\Phi)$ where $\phi:A \to A'$ is a morphism of commutative $\K$-algebras and $\Phi: L \to \cA_A^{A'}(L',\omega')$ is a morphism of Lie-Rinehart algebras over $A$.
\end{definition}

\begin{remark}
If $A = A'$ and $\phi = \id$, then $L'$ itself with $p_1 = \id$ and $p_2 = \omega'$ is a pullback of $\Derk{A} = \Derk{A} \xleftarrow{\omega'} L'$. Therefore, a morphism $(\id,\Phi)$ of Lie-Rinehart algebras from $(A,L,\omega)$ to $(A,L',\omega')$ is the same as a morphism of Lie-Rinehart algebras over $A$ as in \S\ref{ssec:LRalg}.
\end{remark}

Recall that in \cite[page 61]{Huebschmann-Poisson} a morphism of Lie-Rinehart algebras from $(A,L,\omega)$ to $(A',L',\omega')$ is defined as a pair $(\phi,\psi)$ where $\phi:A \to A'$ is a morphism of $\K$-algebras and $\psi : L \to L'$ is a morphism of Lie algebras and of left $A$-modules such that for all $a \in A$, $X \in L$,
\begin{equation}\label{eq:morphism}
\phi\big(\omega(X)(a)\big) = \omega'\big(\psi(X)\big)\big(\phi(a)\big).
\end{equation}

\begin{proposition}\label{prop:LRmorph}
The datum of a morphism of Lie-Rinehart algebras as in Definition \ref{def:morph} is equivalent to the datum of a morphism in the sense of \cite[page 61]{Huebschmann-Poisson}.
\end{proposition}

\begin{proof}
By the universal property of the pullback, giving a morphism $\Psi : L \to \cA^{A'}_A(L',\omega')$ of Lie-Rinehart algebras over $A$ is equivalent to giving a morphism of Lie algebras and of left $A$-modules $\psi:L \to L'$ such that $\phi^*\circ \omega'\circ \psi = \phi_* \circ \omega$,
which is exactly \eqref{eq:morphism}.
\end{proof}


\subsection{Modules over Lie-Rinehart algebras}\label{ssec:mods}
Recall that an $(A,L,\omega)$-module in the sense of \cite[page 62]{Huebschmann-Poisson} is a left $A$-module $M$ together with a morphism of Lie algebras $\rho:L \to \End{\K}{M}$ such that
\[
\rho(a\cdot X)(m) = a\cdot \rho(X)(m) \qquad \text{and} \qquad \rho(X)(a\cdot m) = a\cdot \rho(X)(m) + \omega(X)(a)\cdot m
\]
for all $X \in L$, $m \in M$, $a \in A$.

\begin{proposition}\label{prop:mods}
For a left $A$-module $M$, the datum of a left $(A,L,\omega)$-module structure as in \cite[page 62]{Huebschmann-Poisson} is equivalent to the datum of a morphism $L \to \cL_A\big(\End{\K}{M}\big)$ of Lie-Rinehart algebras over $A$.
\end{proposition}

\begin{proof}
Completely analogous to the proof of Proposition \ref{prop:modsanch} and Proposition \ref{prop:LRmorph}.
\end{proof}

As a consequence, Theorem \ref{thm:main2} and Proposition \ref{prop:mods} provide a conceptual proof of the well-known equivalence between the category of $(A,L,\omega)$-modules and the category of $\cU_A(L)$-modules (see \cite[page 65]{Huebschmann-Poisson}).


\subsection{The infinitesimal gauge algebra of a module} Let $(L,\omega)$ be a Lie-Rinehart algebra over $A$ and let $M$ be an $A$-module. In \cite[page 72]{Huebschmann-Poisson}, a Lie-Rinehart algebra $\mathrm{DO}(A,L,M)$ is introduced, which acts on $M$ by the analogue of infinitesimal gauge transformations. We show how $\mathrm{DO}(A,L,M)$ can naturally be obtained via the constructions we performed and, as a consequence, how it naturally inherits a universal property as well.

The following is the analogue of Proposition \ref{prop:prod} for Lie-Rinehart algebras.

\begin{proposition}
Let $(L,\omega)$ and $(L',\omega')$ be Lie-rinehart algebras over $A$. The product of $(L,\omega)$ and $(L',\omega')$ in $\LieRin_{A}$ exists and can be computed as the pullback \eqref{eq:prod} with component-wise bracket and $A$-action and with anchor $\omega_\times \coloneqq \omega\circ q_1 = \omega'\circ q_2$.
\end{proposition}

\begin{proof}
We already know that, with the structures of the statement, $L \times_{\Derk{A}} L'$ is an $A$-anchored Lie algebra and it is clearly a left $A$-module via the component-wise $A$-action. The anchor $\omega_\times$, being the composition of $A$-linear maps, is $A$-linear. We are left to check the Leibniz condition \eqref{eq:Leibniz}. Since for every $a \in A$ and $(X,Y) \in L \times_{\Derk{A}} L'$ we have $\omega(X)(a) = \omega'(Y)(a) = \omega_\times(X,Y)(a)$, the following direct computation concludes the proof:
\begin{align*}
\big[(X,Y),a\cdot(X',Y')\big] & = \big([X,a\cdot X'],[Y,a\cdot Y']\big) \\
 & \stackrel{\eqref{eq:Leibniz}}{=} \big(a\cdot [X,X'] + \omega(X)(a)\cdot X', a\cdot [Y,Y'] + \omega'(Y)(a)\cdot Y'\big) \\
 & = a\cdot \big[(X,Y),(X',Y')\big] + \omega_\times(X,Y)(a)\cdot(X',Y'). \qedhere
\end{align*}
\end{proof}

Let $(A,L,\omega)$ be a Lie-Rinehart algebra. Recall from \cite[page 72]{Huebschmann-Poisson} that for a given $A$-module $M$, the Lie-Rinehart algebra $\mathrm{DO}(A,L,M)$ of infinitesimal gauge transformations of $M$ with respect to $L$ is the subspace of $\End{\K}{M} \times L$ composed by the elements $(f,X)$ such that
\[f(a\cdot m) = \omega(X)(a)\cdot m + a\cdot f(m)\]
for all $a\in A$, $m \in M$. The bracket and the $A$-action are given component-wise, while the anchor $\tilde{\omega}$ is induced by the restriction of the projection on the second factor. 

\begin{proposition}\label{prop:DO}
Let $M$ be an $A$-module and let $(A,L,\omega)$ be a Lie-Rinehart algebra. The Lie-Rinehart algebra $\big(A,\mathrm{DO}(A,L,M),\tilde{\omega}\big)$ is the product in $\LieRin_A$ of the Lie-Rinehart algebras $(A,L,\omega)$ and $\cL_A\big(\End{\K}{M}\big)$.
\end{proposition}

\begin{proof}
Set $E \coloneqq \End{\K}{M}$ and $\phi: A \to E, a \mapsto l_a$. By definition of $\cL_A(E)$ and by the construction of the product in $\LieRin_A$, $(A,L,\omega) \times \cL_A(E)$ is the following pasting of pullbacks:
\[
\xymatrix @R=14pt @C=17pt{
\cL_A(E) \underset{\Derk{A}}{\times} L  \ar@{}[dr]|(0.35){\Big\lrcorner} \ar[r]^-{q_1} \ar[dd]_-{q_2} & \cL_A(E)  \ar@{}[dr]|(0.35){\Big\lrcorner} \ar[dd]^-{p_2} \ar[r]^-{p_1} & E \ar[d]^-{\varpi_E} \\
 & & \Derk{E} \ar[d]^-{\phi^*} \\
L \ar[r]_-{\omega} & \Derk{A} \ar[r]_-{\phi_*} & \Derk{A,E}.
}
\]
Concretely,
$ 
\cL_A(E) \underset{\Derk{A}}{\times} L = \big\{(f,X) \in E \times L\mid [f,l_a] = l_{\omega(X)(a)} \text{ for all }a \in A\big\}
$ 
with component-wise bracket and $A$-action and with anchor given by $p_2 \circ q_1 = \omega \circ q_2 = \tilde{\omega}$.
\end{proof}

It follows from Proposition \ref{prop:mods} that $M$ is an $\big(A,\mathrm{DO}(A,L,M),\tilde{\omega}\big)$-module via $q_1$. Namely, $\varrho \coloneqq (p_1 \circ q_1) : \mathrm{DO}(A,L,M) \to \End{\K}{M}$ satisfies the conditions of \S\ref{ssec:mods}. Furthermore, $\mathrm{DO}(A,L,M)$ admits a canonical Lie-Rinehart algebra morphism $q_2 : \mathrm{DO}(A,L,M) \to L$ and, in fact, $\big(A,\mathrm{DO}(A,L,M),\tilde{\omega}\big)$ is universal with respect to these properties.

\begin{theorem}\label{thm:DO}
Let $M$ be an $A$-module and let $(A,L,\omega)$ be a Lie-Rinehart algebra. For every Lie-Rinehart algebra $(A,L',\omega')$ acting on $M$ via $\rho: L'\to \End{\K}{M}$ and any morphism of Lie-Rinehart algebras $f : L'\to L$, 
there exists a unique morphism $\tilde{f} : L'\to \mathrm{DO}(A,L,M)$ of Lie-Rinehart algebras over $A$ such that $\varrho \circ \tilde{f} = \rho$ and $q_2 \circ \tilde{f} = f$.
\end{theorem}

\begin{proof}
It follows from Proposition \ref{prop:mods} and Proposition \ref{prop:DO}.
\end{proof}

\begin{example}
Among all infinitesimal gauge algebras $\mathrm{DO}(A,L,M)$ associated with an $A$-module $M$, there exists a universal one, which is the \emph{Atiyah algebra} $\cA_M$ of $M$ (see \eg \cite[(1.1.3) Examples (c)]{Kapranov}). This is the Lie-Rinehart algebra $\cL_A(\End{\K}{M})$ of infinitesimal gauge transformations of $M$ with respect to $\Derk{A}$. Concretely,
\[\cA_M = \big\{(f,\delta) \in \End{\K}{M} \times \Derk{A} ~\big|~ f(a\cdot m) = a\cdot f(m) + \delta(a)\cdot m\big\}.\]
By Theorem \ref{thm:DO}, if $(A,L,\omega)$ is a Lie-Rinehart algebra acting on $M$ via $\rho: L\to \End{\K}{M}$, then there exists a unique morphism $\tau_L:L\to \cA_M$ of Lie-Rinehart algebras over $A$ such that $p_1 \circ \tau_L = \rho$. In particular, the datum of a left $(A,L,\omega)$-module structure on $M$ is equivalent to the datum of a morphism $L \to \cA_M$ of Lie-Rinehart algebras over $A$ (see \cite[(1.1.4) Definition]{Kapranov} and Proposition \ref{prop:mods}).
\end{example}

%
%
%
%



\begin{thebibliography}{99}


\bibitem{ArdiLaiachiPaolo}
A. Ardizzoni, L. El Kaoutit, P. Saracco, \emph{Differentiation and integration between Hopf algebroids and Lie algebroids}. Preprint (2019). (\href{https://arxiv.org/abs/1905.10288}{arXiv: 1905.10288})

%
%
%

\bibitem{Bourbaki}
N. Bourbaki, \emph{\'El\'ements de math\'ematique. Groupes et alg\`ebres de Lie. Chapitre 1. Reprint of the 1972 original}. Berlin: Springer. 148 p. (2007).

%
%
%

\bibitem{dascalescu} 
S.~D\u{a}sc\u{a}lescu, C.~N\u{a}st\u{a}sescu, \c{S}.~Raianu, \emph{Hopf algebras. An introduction}. Monographs and Textbooks in Pure and Applied Mathematics, \textbf{235}. Marcel Dekker, Inc., New York, 2001.


\bibitem{LaiachiPaolo-diff}
L. El Kaoutit, P. Saracco, \emph{The Hopf Algebroid Structure of Differentially Recursive Sequences}. Preprint (2020). (\href{https://arxiv.org/abs/2003.08180v1}{arXiv: 2003.08180})

\bibitem{LaiachiPaolo-complete}
L.~El~Kaoutit, P.~Saracco, \emph{Topological tensor product of bimodules, complete Hopf algebroids and convolution algebras}.  Commun. Contemp. Math. \textbf{21} (2019), no. 6, 1850015, 53 pp.

%
%
%
%

\bibitem{Huebschmann-quantization}
J.~Huebschmann, \emph{Lie-Rinehart algebras, descent, and quantization}. Galois theory, Hopf algebras, and semiabelian categories, 295--316, Fields Inst. Commun., \textbf{43}, Amer. Math. Soc., Providence, RI, 2004. 

\bibitem{Huebschmann-LR}
J.~Huebschmann, \emph{Lie-Rinehart algebras, Gerstenhaber algebras and Batalin-Vilkovisky algebras}. Ann. Inst. Fourier (Grenoble) \textbf{48} (1998), no. 2, 425--440.  

\bibitem{Huebschmann-Poisson}
J.~Huebschmann, \emph{Poisson cohomology and quantization}. J. Reine Angew. Math. \textbf{408} (1990), 57--113.

\bibitem{Jacobson}
N. Jacobson, \emph{Lie algebras}. Interscience Tracts in Pure and Applied Mathematics, No. \textbf{10} Interscience Publishers (a division of John Wiley \& Sons), New York-London 1962.

%

\bibitem{Kapranov}
M.~Kapranov, \emph{Free Lie algebroids and the space of paths}. Selecta Math.~(N.S.) \textbf{13} (2007), no.~2, 277--319.


\bibitem{MacLane} 
S.~Mac~Lane, \emph{Categories for the Working Mathematician. Second edition}. Graduate Texts in Mathematics, \textbf{5}. Springer-Verlag, New York, 1998.


\bibitem{Malliavin}
M.-P. Malliavin, \emph{Alg\`ebre homologique et op\'erateurs diff\'erentiels}. Ring theory (Granada, 1986), 173--186, Lecture Notes in Math., \textbf{1328}, Springer, Berlin, 1988.

%

\bibitem{MoerdijkLie} 
I.~Moerdijk, J.~Mr\v{c}un, \emph{On the universal enveloping algebra of a Lie algebroid}. Proc. Amer. Math. Soc. \textbf{138} (2010), no. 9, 3135--3145.

%
%
%
%
%
%
%

\bibitem{Rinehart} 
G.~S.~Rinehart, \emph{Differential forms on general commutative algebras}. Trans. Amer. Math. Soc. \textbf{108}, 1963, 195--222.

\bibitem{Saracco-anch}
P.~Saracco, \emph{On anchored Lie algebras and the Connes-Moscovici's bialgebroid construction}. Preprint (2020). (\href{https://arxiv.org/abs/2009.14656}{arXiv: 2009.14656})

%

\bibitem{Sweedler-groups}
M.~E.~Sweedler, \emph{Groups of simple algebras}. Inst. Hautes \'Etudes Sci. Publ. Math. No. \textbf{44} (1974), 79--189.

%

\end{thebibliography}
\end{document}